\documentclass[preprint,11pt,sort&compress,numbers,draft]{elsarticle}

\usepackage{amsmath,amsthm,amsfonts,amssymb,latexsym,mathrsfs,color,cases,url,enumerate}
\journal{arXiv}

\newtheorem{thm}{Theorem}[section]
\newtheorem{lem}[thm]{Lemma}
\newtheorem{coro}[thm]{Corollary}

\newtheorem{prop}[thm]{Proposition}
\newtheorem*{SE}{Schoenberg-Edrei Theorem}

\newtheorem*{CBF}{Cauchy-Binet Formula}
\newtheorem*{BKW}{Beraha-Kahane-Weiss Theorem}

\newtheorem*{NI}{Newton Inequality}
\newtheorem*{DJI}{Desnanot-Jacobi Determinant Identity}
\theoremstyle{definition}

\newtheorem{rem}[thm]{Remark}
\newtheorem{exm}[thm]{Example}

\newcommand{\lrf}[1]{\left\lfloor #1\right\rfloor}
\newcommand{\lrc}[1]{\left\lceil #1\right\rceil}
\newcommand{\lmn}[2]{\ell_{#2}^{(#1)}}
\newcommand\lm[1]{\ell^{(#1)}}
\def\R{\mathcal{R}}
\def\M{\mathcal{M}}
\def\L{\mathcal{L}}
\def\s{\widehat{S}}
\def\p{\widehat{P}}

\def\d{\widehat{D}}
\def\c{\widehat{C}}
\def\hr{\widehat{R}}

\def\r{R^+}

\def\sint{\prec_{int}}
\def\salt{\prec_{alt}}
\def\sgn{\mathrm{sgn\,}}
\def\la{\lambda}
\def\N{\mathbb{N}}
\def\Z{\mathbb{Z}}
\def\G{\mathfrak{G}}

\def\rz{\mathrm{RZ}}
\def\2p{\Lambda}

\numberwithin{equation}{section}

\begin{document}

\begin{frontmatter}

\title{Analytic combinatorics of coordination numbers of cubic lattices}
\author[a]{Huyile Liang\corref{cor1}}
\ead{lianghuyile@imnu.edu.cn}
\author[b]{Yanni Pei\corref{cor2}}
\ead{peiyanni@hotmail.com}
\author[b]{Yi Wang\corref{cor3}}
\ead{wangyi@dlut.edu.cn}
\cortext[cor3]{Corresponding author.}
\address[a]{College of Mathematics Science and Center for Applied Mathematical Science,
Inner Mongolia Normal University, Hohhot 010022, P.R. China}
\address[b]{School of Mathematical Sciences, Dalian University of Technology, Dalian 116024, P.R. China}
\date{}

\begin{abstract}
We investigate coordination numbers of the cubic lattices
with emphases on their analytic behaviors,
including the total positivity of the coordination matrices,
the distribution of zeros of the coordination polynomials,
the asymptotic normality of the coefficients of the coordination polynomials,
the log-concavity and the log-convexity of the coordination numbers.
\end{abstract}


\end{frontmatter}

\section{Introduction}

Following Conway and Sloane \cite{CS97},
the {\it coordination sequence} of an infinite vertex-transitive graph $\G$ is the sequence
$(S(0), S(1), S(2),\ldots)$,
where $S(n)$ is the number of vertices at distance $n$ from some fixed vertex of $\G$.
The partial sums
$D(n)=S(0)+S(1)+\cdots+S(n)$
are called the {\it crystal ball numbers}.
As was pointed out in O'Keeffe \cite{OK98},
one can use the coordination sequence somewhat like a fingerprint to identity structures of $\G$.
We refer the reader to \cite{BG97,CS97,OK98} and references therein for details.
Let $S(x)=\sum_{n\ge 0}S(n)x^n$
and $D(x)=\sum_{n\ge 0}D(n)x^n$
be the generating functions of the coordination sequence and the crystal ball numbers.
Then $D(x)=S(x)/(1-x)$.

For the $k$-dimensional integer lattice $\Z^k$,
Conway and Sloane \cite{CS97} gave the generating functions
$S_k(x)=(1+x)^k/(1-x)^k$
and
$D_k(x)=(1+x)^k/(1-x)^{k+1}$
of the coordination sequence and the crystal ball numbers respectively.
Denote $S_k(x)=\sum_{n\ge 0}S(n,k)x^n$ and $D_k(x)=\sum_{n\ge 0}D(n,k)x^n$.
The first few terms of the coordination numbers $S(n,k)$
and crystal ball numbers $D(n,k)$ of the cubic lattices are as follows:
\begin{center}
\begin{tabular}{ccc}
\begin{tabular}{c|*{5}{c}}
$n\setminus k$
 & 0 & 1 & 2 & 3 & 4 \\
\hline
0 & 1 & 1 & 1 & 1 & 1\\
1 & 0 & 2 & 4 & 6 & 8\\
2 & 0 & 2 & 8 & 18 & 32\\
3 & 0 & 2 & 12 & 38 & 88\\
4 & 0 & 2 & 16 & 66 & 192\\
\end{tabular}
& &
\begin{tabular}{c|*{5}{c}}
$n\setminus k$
 & 0 & 1 & 2 & 3 & 4 \\
\hline
0 & 1 & 1 & 1 & 1& 1 \\
1 & 1 & 3 & 5 & 7& 9\\
2 & 1 & 5 & 13 & 25& 41\\
3 & 1 & 7 & 25 & 63& 129\\
4 & 1 & 9 & 41 & 129& 321 \\
\end{tabular}
\\[15mm] The numbers $S(n,k)$ & & The numbers $D(n,k)$
\end{tabular}
\end{center}

By the definition, we have
\begin{equation*}\label{S-def}
S(n,k)=\#\{(x_1,\ldots,x_k)\in\Z^k:\quad |x_1|+\cdots+|x_k|=n\}
\end{equation*}
and
\begin{equation}\label{D-S}
D(n,k)=S(0,k)+S(1,k)+\cdots+S(n,k).
\end{equation}
Let $S(x,y)=\sum_{n,k\ge 0}S(n,k)x^ny^k$
and $D(x,y)=\sum_{n,k\ge 0}D(n,k)x^ny^k$
denote the bivariate generating functions of $S(n,k)$ and $D(n,k)$ respectively.
Then
\begin{equation}\label{Sxy}
  S(x,y)=\sum_{k\ge 0}S_k(x)y^k=\sum_{k\ge 0}\left(\frac{1+x}{1-x}\right)^ky^k=\frac{1-x}{1-x-y-xy}
\end{equation}
and
\begin{equation}\label{Dxy}
  D(x,y)=\sum_{k\ge 0}D_k(x)y^k=\sum_{k\ge 0}\frac{(1+x)^k}{(1-x)^{k+1}}y^k
  =\frac{1}{1-x}S(x,y)=\frac{1}{1-x-y-xy}.
\end{equation}

It is clear from \eqref{Dxy} that
the numbers $D(n,k)$ turn out to be the {\it Delannoy numbers}
(see \cite{WZC19} for instance and \cite{BS05} for historical remarks).
Let $D=[D(n,k)]_{n,k\ge 0}$ be the infinite Delannoy matrix.
Then the generating function of the $k$th column of $D$ is $D_k(x)$ for $k=0,1,2,\ldots$.
We denote it by
$$D=\M[D_0(x),D_1(x),D_2(x),\ldots].$$
Let $\d=\M[D_0(x),xD_1(x),x^2D_2(x),\ldots]=[d(n,k)]_{n,k\ge 0}$
be the infinite lower triangular matrix corresponding to $D$.
Then $d(n,k)=D(n-k,k)$ for $n\ge k\ge 0$.
Let $d_n(x)=\sum_{k=0}^nd(n,k)x^k$ be the $n$th row generating function of $\d$.
We call $\d$ and $d_n(x)$ the {\it Delannoy triangle} and the {\it Delannoy polynomials} respectively.

We define four matrices related to the coordination numbers of the cube lattices as follows:
\begin{enumerate}[\rm (i)]
  \item $S=\M[S_0(x),S_1(x),S_2(x),\ldots]=[S(n,k)]_{n,k\ge 0}$.
  \item $C=\M[S_1(x),S_2(x),S_3(x),\ldots]=[C(n,k)]_{n,k\ge 0}$.
  \item $\s=\M[S_0(x),xS_1(x),x^2S_2(x),\ldots]=[s(n,k)]_{n,k\ge 0}$.
  \item $\c=\M[S_1(x),xS_2(x),x^2S_3(x),\ldots]=[c(n,k)]_{n,k\ge 0}$.
\end{enumerate}
Clearly, $C(n,k)=S(n,k+1)$ for $n,k\ge 0$,
$s(n,k)=S(n-k,k)$ and $c(n,k)=C(n-k,k)$ for $n\ge k\ge 0$.
Let $C(x,y)=\sum_{n,k\ge 0}C(n,k)x^ny^k$ be the bivariate generating functions of $C(n,k)$.
Then
\begin{equation}\label{Cxy}
  C(x,y)=\sum_{n,k\ge 0}S(n,k+1)x^ny^k=\sum_{k\ge 0}S_{k+1}(x)y^k
  =\sum_{k\ge 0}\left(\frac{1+x}{1-x}\right)^{k+1}y^k=\frac{1+x}{1-x-y-xy}.
\end{equation}
Let $c_n(x)=\sum_{k=0}^nc(n,k)x^k$ and $s_n(x)=\sum_{k=0}^ns(n,k)x^k$
be the $n$th row generating function of $\c$ and $\s$ respectively.
Note that
\begin{equation*}\label{c-tri}
  \c=[c(n,k)]_{n\ge k\ge 0}=
  \left(\begin{array}{cccccc}
  1 &  &  &  &  &  \\
  2 & 1 &  &  &  &  \\
  2 & 4 & 1 &  &  &  \\
  2 & 8 & 6 & 1 &  &  \\
  2 & 12 & 18 & 8 & 1 &  \\
  \vdots &  &  &  &  & \ddots \\
 \end{array}\right),\qquad
\s=\left(\begin{array}{cc}
1 & 0 \\
0 & \c \\
\end{array}\right)
\end{equation*}
and $s_n(x)=xc_{n-1}(x)$ for $n\ge 1$.
For our purpose,
we call the matrix $\c$ and its row generating functions $c_n(x)$
the {\it coordination triangle} and the {\it coordination polynomials} of the cubic lattices respectively.
We refer the reader to distinguish the coordination triangle
from the coordinator triangle defined by Conway and Sloane \cite{CS97}.
The coordinator triangle of the cube lattices is the Pascal triangle.
It is interesting that if we substitute the nonnegative integer lattice $\N$ for the integer lattice $\Z$,
then the corresponding coordination matrices $C$ and $\c$
will be the Pascal square $P=\left[\binom{n+k}{k}\right]_{n,k\ge 0}$
and the Pascal triangle $\p=\left[\binom{n}{k}\right]_{n,k\ge 0}$ respectively.

The main objective of this paper is to investigate
analytic properties of the coordination numbers of the cubic lattices.
The paper is organized as follows.
In \S 2,
we provide some preliminary results on the coordination numbers,
the coordination triangle and the coordination polynomials.
In \S 3,
we consider analytic behaviors of the coordination numbers,
including the total positivity of the coordination matrices,
the distribution of zeros of the coordination polynomials,
the asymptotic normality of the coefficients of the coordination polynomials,
the log-concavity and the log-convexity of the coordination numbers.
In \S 4,
we discuss further work and related problems.


\section{Preliminaries}

In this section we present some basic properties of coordination numbers of the cubic lattices.
Most of them are elementary and partial results have occurred in the literature.
We may investigate the coordination numbers and the crystal ball numbers of the cube lattices
in a unified approach.
Actually, let $m\in\N$.
Define the numbers $L^{(m)}(n,k)$ by
\begin{equation}\label{Lnk-rr}
  L^{(m)}(n,k)=L^{(m)}(n-1,k-1)+L^{(m)}(n-1,k)+L^{(m)}(n,k-1),
\end{equation}
with the initial values $L^{(m)}(0,k)=1$ for $k\ge 0$ and $L^{(m)}(n,0)=m$ for $n\ge 1$.
It follows from \eqref{Lnk-rr} that the bivariate generating function of $L^{(m)}(n,k)$ is
\begin{equation}\label{Lxy}
  L^{(m)}(x,y)=\sum_{n,k\ge 0}L^{(m)}(n,k)x^ny^k=\frac{1+(m-1)x}{1-x-y-xy}.
\end{equation}
Comparing \eqref{Lxy} to \eqref{Sxy}, \eqref{Dxy} and \eqref{Cxy} respectively,
we see that the Delannoy numbers $D(n,k)=L^{(1)}(n,k)$,
the coordination numbers $S(n,k)=L^{(0)}(n,k)$ and $C(n,k)=L^{(2)}(n,k)$.
Note also that $L^{(m)}(x,y)=[1+(m-1)x]D(x,y)$.
Hence
\begin{equation}\label{L-D}
L^{(m)}(n,k)=D(n,k)+(m-1)D(n-1,k)
\end{equation}
for $n,k\ge 0$.
Here we use the notation $D(n,k)=0$ unless $n,k\ge 0$.

\subsection{Coordination numbers and Delannoy numbers}

The Delannoy numbers enjoy nice combinatorial interpretations and beautiful formulas.
For example,
the Delannoy numbers $D(n,k)$ count the number of lattice paths from $(0,0)$ to $(n,k)$
using steps $(1,0), (0,1)$ and $(1,1)$.
It is well known \cite[p. 81]{Com74} that
\begin{equation}\label{Dnk-2}
  D(n,k)=\sum_{i}\binom{n}{i}\binom{k}{i}2^i
\end{equation}
and
\begin{equation}\label{Dnk}
  D(n,k)=\sum_{i}\binom{n+i}{k}\binom{k}{i}.
\end{equation}
The following results are immediate from \eqref{L-D}, \eqref{Dnk-2} and \eqref{Dnk}.

\begin{prop}\label{C-N}
We have
\begin{enumerate}[\rm (i)]
\item
$S(n,k)=D(n,k)-D(n-1,k)$.
\item
$C(n,k)=D(n,k)+D(n-1,k)$.
\item
$S(n,k)=\sum_{i}\binom{n-1}{i-1}\binom{k}{i}2^i$.
\item
$S(n,k)=\sum_{i}\binom{n+i-1}{k-1}\binom{k}{i}$.
\end{enumerate}
\end{prop}

The central Delannoy numbers $D_n=D(n,n)$ are
the main diagonal of the Delannoy matrix $D$.
The first few terms of the sequence $(D_n)_{n\ge 0}$ are
$1,3,13,63,321,\ldots$~\cite[A001850]{Slo}.
Sulanke~\cite{Sul03} listed 29 objects counted by the central Delannoy numbers.
The central Delannoy numbers are closely related to the Jacobi polynomials $P_n^{(\alpha,\beta)}(t)$,
which are a class of classical orthogonal polynomials defined by
\begin{equation}\label{pnx-ee}
P_n^{(\alpha,\beta)}(t)
=\sum_{i=0}^n\binom{n+\alpha}{i}\binom{n+\beta}{n-i}
  \left(\frac{t+1}{2}\right)^{i}\left(\frac{t-1}{2}\right)^{n-i}.
\end{equation}
By \eqref{Dnk-2}, the central Delannoy numbers
$D_n=\sum_{i=0}^n\binom{n}{i}^22^i=P_n^{(0,0)}(3)$.

Let $S_n=S(n,n)$ and $C_n=C(n,n)$
be the main diagonals of the coordination matrices $S$ and $C$ respectively.
Then
$(S_n)_{n\ge 0}=(1,2,8,38,192,\ldots)$ \cite[A123164]{Slo}
and
$(C_n)_{n\ge 0}=(1,4,18,88,450,\ldots)$ \cite[A050146]{Slo}.
It follows from Proposition \ref{C-N} (iii)
that
$$S_n=S(n,n)=\sum_{i=0}^{n}\binom{n-1}{i-1}\binom{n}{i}2^i
=\sum_{i=0}^{n}\binom{n}{i}\binom{n-1}{n-i}2^i
=P_n^{(0,-1)}(3)$$
and
\begin{equation*}\label{Cn-ee}
C_n=C(n,n)=S(n,n+1)=\sum_{i=0}^{n}\binom{n-1}{i-1}\binom{n+1}{i}2^i
=\sum_{i=0}^{n}\binom{n+1}{i}\binom{n-1}{n-i}2^i
=P_n^{(1,-1)}(3).
\end{equation*}

It is well known \cite[Chapter 4]{Sze75} that
the Jacobi polynomials satisfy the recurrence relation
\begin{eqnarray}\label{pnx-rr}
 &&2n(n+\alpha+\beta)(2n+\alpha+\beta-2)P_{n}^{(\alpha,\beta)}(t)\nonumber \\
 &=&(2n+\alpha+\beta-1)\left[(2n+\alpha+\beta)(2n+\alpha+\beta-2)t+\alpha^2-\beta^2\right]P_{n-1}^{(\alpha,\beta)}(t)\nonumber\\
 &&-2(n+\alpha-1)(n+\beta-1)(2n+\alpha+\beta)P_{n-2}^{(\alpha,\beta)}(t)
\end{eqnarray}
with $P_0^{(\alpha,\beta)}(t)=1$ and $P_1^{(\alpha,\beta)}(t)=(\alpha+1)+(\alpha+\beta+2)(t-1)/2$.
Also, the Jacobi polynomials have the generating function
\begin{equation}\label{pnx-gf}
  \sum_{n\ge 0}P_n^{(\alpha,\beta)}(t)x^n
  =2^{\alpha+\beta}R^{-1}(1-x+R)^{-\alpha}(1+x+R)^{-\beta},
\end{equation}
where $R=R(t,x)=\sqrt{1-2tx+x^2}$.

\begin{prop}\label{SC-p}
We have
\begin{enumerate}[\rm (i)]
  \item
  $nD_n=3(2n-1)D_{n-1}-(n-1)D_{n-2}$
  with $D_0=1$ and $D_1=3$.
  \item
  $n(2n-3)S_{n}=4(3n^2-6n+2)S_{n-1}-(n-2)(2n-1)S_{n-2}$
  with $S_0=1$ and $S_1=2$.
  \item
  $n(n-1)C_n=3(2n-1)(n-1)C_{n-1}-n(n-2)C_{n-2}$
  with $C_0=1$ and $C_1=4$.
  \item
  $\sum_{n\ge 0}D_nx^n=\dfrac{1}{\sqrt{1-6x+x^2}}$.
  \item
  $\sum_{n\ge 0}S_nx^n=\dfrac{1+x+\sqrt{1-6x+x^2}}{2\sqrt{1-6x+x^2}}$.
  \item
  $\sum_{n\ge 0}C_nx^n=\dfrac{3-x-\sqrt{1-6x+x^2}}{2\sqrt{1-6x+x^2}}$.
\end{enumerate}
\end{prop}

\begin{rem}
From the generating functions of $(S_n)_{n\ge 0}$ and $(D_n)_{n\ge 0}$,
it is easy to see that
\begin{equation}\label{s-d-d}
  S_n=\frac{1}{2}(D_n+D_{n-1})
\end{equation}
for $n\ge 1$.
\end{rem}

\begin{rem}
The coordination number $C_n$ is closely related to the large Schr\"oder number $r_n$,
which counts the number of subdiagonal paths from $(0,0)$ to $(n,n)$
consisting of steps $(1,0),(0, 1)$ and $(1,1)$.
The first few terms of the sequence $(r_n)_{n\ge 0}$ are $1,2,6,22,90,\ldots$ \cite[A006318]{Slo}.
It is well known \cite{SS91} that the large Schr\"oder numbers have
the generating function
\begin{equation*}\label{rn-gf}
\sum_{n\ge 0}r_nx^n=\frac{1-x-\sqrt{1-6x+x^2}}{2x}.
\end{equation*}
Let $r(x)$ and $C(x)$ be the generating functions of $(r_n)_{n\ge 0}$ and $(C_n)_{n\ge 0}$ respectively.
It is not difficult to check that $[xr(x)]'=C(x)$.
Hence
\begin{equation}\label{c-r}
C_n=(n+1)r_n.
\end{equation}
In other words,
the coordination numbers $C_n$ bear the same relation to the large Schr\"oder numbers
as the central binomial coefficients $\binom{2n}{n}$ do to the Catalan numbers $\frac{1}{n+1}\binom{2n}{n}$,
the so-called Chung-Feller property.
An interesting problem is to find out a combinatorial interpretation for \eqref{c-r}.
\end{rem}

\subsection{Coordination matrices as Riordan arrays}

Let $f(x)=\sum_{n\ge 0}f_nx^n$ and $g(x)=\sum_{n\ge 0}g_nx^n$
be two formal power series.
A {\it Riordan array},
denoted by $\R(g(x),f(x))$,
is an infinite matrix whose generating function of the $k$th column is $g(x)f^k(x)$ for $k\ge 0$:
$$\R(g(x),f(x))=\M[g(x),g(x)f(x),g(x)f^2(x),\ldots].$$
Riordan arrays play an important unifying role in enumerative combinatorics
\cite{SGWW91,Spr94}.

\begin{exm}
\begin{enumerate}[\rm (i)]
  \item
  The Pascal square $P=\R\left(\frac{1}{1-x},\frac{1}{1-x}\right)$
  and the Pascal triangle $\p=\R\left(\frac{1}{1-x},\frac{x}{1-x}\right)$.
  \item
  Let $g(x)=\sum_{n\ge 0}g_nx^n$.
  Then the Toeplitz matrix
  $$T(g)=[g_{i-j}]_{i,j\ge 0}
  =\left(
  \begin{array}{ccccc}
    g_0 & 0 & 0 & 0 & \cdots \\
    g_1 & g_0 & 0 & 0 &  \\
    g_2 & g_1 & g_0 & 0 &  \\
    g_3 & g_2 & g_1 & g_0 &  \\
    \vdots &  &  &  & \ddots \\
  \end{array}
\right)$$
of the sequence $(g_n)_{n\ge 0}$ is a Riordan array: $T(g)=\R(g(x),x)$.
\end{enumerate}
\end{exm}

Let $L^{(m)}=[L^{(m)}(n,k)]_{n,k\ge 0}$ and
let $L_k^{(m)}(x)=\sum_{n\ge 0}L^{(m)}(n,k)x^n$ be the $k$th column generating function of the matrix $L^{(m)}$.
Then by \eqref{Lxy}, we have
\begin{equation*}\label{Lnx}
  L_k^{(m)}(x)=[y^k]L^{(m)}(x,y)
  =\frac{1+(m-1)x}{1-x}\cdot\left(\frac{1+x}{1-x}\right)^k.
\end{equation*}
Thus $L^{(m)}$ is a Riordan array:
\begin{equation*}\label{L-RA}
L^{(m)}=\R\left(\frac{1+(m-1)x}{1-x},\frac{1+x}{1-x}\right).
\end{equation*}
The corresponding lower triangular matrix $\widehat{L^{(m)}}=[L^{(m)}(n-k,k)]_{n,k\ge 0}$
is also a Riordan array:
\begin{equation*}\label{LT-RA}
\widehat{L^{(m)}}=\R\left(\frac{1+(m-1)x}{1-x},\frac{x(1+x)}{1-x}\right).
\end{equation*}
In particular,
the Delannoy matrix $D=\R\left(\frac{1}{1-x},\frac{1+x}{1-x}\right)$
and the Delannoy triangle $\d=\R\left(\frac{1}{1-x},\frac{x(1+x)}{1-x}\right)$.

\begin{prop}\label{SC-R}
The coordination matrices $S,\s,C,\c$ are all Riordan arrays:
\begin{enumerate}[\rm (i)]
  \item
  $S=\R\left(1,\frac{1+x}{1-x}\right)$ and
  $\s=\R\left(1,\frac{x(1+x)}{1-x}\right)$.
  \item
  $C=\R\left(\frac{1+x}{1-x},\frac{1+x}{1-x}\right)$ and
  $\c=\R\left(\frac{1+x}{1-x},\frac{x(1+x)}{1-x}\right)$.
\end{enumerate}
\end{prop}

We say that a Riordan array $\R(g(x),f(x))$ is {\it proper}
if $g_0=1, f_0=0$ and $f_1\neq 0$.
In this case,
$\R(g(x),f(x))$ is an infinite lower triangular matrix.
It is well known \cite{SGWW91} that the set of proper Riordan arrays is a group under the matrix multiplication and
\begin{eqnarray*}\label{rp}
\R(d(x),h(x))\cdot\R(g(x),f(x))=\R(d(x)g(h(x)),f(h(x))).
\end{eqnarray*}
The identity matrix can be written as $\R(1,x)$
and the inverse of $\R(g(x),f(x))$ is given by
$$\R(1/g(\bar{f}(x)),\bar{f}(x)),$$
where $\bar{f}$ is the compositional inverse of $f$,
i.e., $f(\bar{f}(x))=\bar{f}(f(x))=x$.

It is also well known \cite{HS09} that a proper Riordan array $R=[R_{n,k}]_{n,k\ge 0}$
can be characterized by two sequences
$(a_n)_{n\ge 0}$ and $(z_n)_{n\ge 0}$ such that
\begin{equation}\label{rrr-c}
R_{0,0}=1,\quad R_{n+1,0}=\sum_{j\ge 0}z_jR_{n,j},\quad R_{n+1,k+1}=\sum_{j\ge 0}a_jR_{n,k+j}
\end{equation}
for $n,k\ge 0$.
Let $A(x)=\sum_{n\ge 0}a_nx^n$ and $Z(x)=\sum_{n\ge 0}z_nx^n$.
Then
\begin{equation*}\label{gf-az}
g(x)=\frac{1}{1-xZ(f(x))},\quad f(x)=xA(f(x)).
\end{equation*}

By means of the standard techniques in the theory of Riordan arrays,
the $A$- and $Z$- sequences of the coordination triangle $\c$ can be decided by
\begin{equation*}\label{c-ax}
A(x)=\frac{1+x+\sqrt{1+6x+x^2}}{2}=1+2x-2x^2+6x^3-22x^4+90x^5-\cdots 
\end{equation*}
and
\begin{equation*}\label{c-zx}
Z(x)=\frac{-1+x+\sqrt{1+6x+x^2}}{2x}=2-2x+6x^2-22x^3+90x^4-\cdots.
\end{equation*}
In other words, $a_{n+1}=z_{n}=(-1)^{n}r_{n}$ for $n\ge 1$,
where $r_n$ are the large Schr\"oder numbers.
Also, the inverse of $\c$ is still a Riordan array:
$$\c^{-1}=\R(r(-x),xr(-x))
=\left(\begin{array}{rrrrrr}
  1 &  &  &  &  &  \\
  -2 & 1 &  &  &  &  \\
  6 & -4 & 1 &  &  &  \\
  -22 & 16 & -6 & 1 &  &  \\
  90 & -68 & 30 & -8 & 1 &  \\
  \vdots &  &  &  &  & \ddots \\
 \end{array}\right).
$$
The leftmost column in $\c^{-1}$ precisely consists of the signed large Schr\"oder numbers $(-1)^nr_n$.

\subsection{Decomposition of coordination matrices}

Let $[R_{n,k}]_{n\ge k\ge 0}$ be a proper Riordan array
with the $A$- and $Z$- sequences $(a_n)_{n\ge 0}$ and $(z_n)_{n\ge 0}$ respectively.
Then \eqref{rrr-c} is equivalent to
$$
\left(
  \begin{array}{ccccc}
    R_{10} & R_{11} &  &  &  \\
    R_{20} & R_{21} & R_{22} &  &  \\
    R_{30} & R_{31} & R_{32} & R_{33} &  \\
    \vdots &  & \cdots &  & \ddots \\
  \end{array}
\right)
=
\left(
  \begin{array}{cccc}
    R_{00} &&&\\
    R_{10} & R_{11} &  &  \\
    R_{20} & R_{21} & R_{22} &   \\
    \vdots &  &    & \ddots \\
  \end{array}
\right)\cdot
\left(
  \begin{array}{ccccc}
    z_0 & a_0 &  &  &  \\
    z_1 & a_1 & a_0 &  &  \\
    z_2 & a_2 & a_1 & a_0 &  \\
    \vdots & \vdots &  &  & \ddots \\
  \end{array}
\right).$$
The rightmost matrix
$\M[Z(x),A(x),xA(x),x^2A(x),\ldots]$
is sometimes called the {\it product matrix} of the Riordan array.

Let $R=\R(g(x),f(x))$ be a Riordan array
with $g(x)=\sum_{n\ge 0}g_nx^n$ and $f(x)=\sum_{n\ge 0}f_nx^n$.
Denote
\begin{equation*}\label{lpm}
\L(g(x),f(x))=
\left(
  \begin{array}{ccccc}
    g_0 & f_0 &  &  &  \\
    g_1 & f_1 & f_0 &  &  \\
    g_2 & f_2 & f_1 & f_0 &  \\
    \vdots & \vdots &  &  & \ddots \\
  \end{array}
\right)
\end{equation*}
and
$$\r=
\left(\begin{array}{cc}
1 & 0 \\
0 & R \\
\end{array}\right).
$$
We call $\L(g(x),f(x))$ the {\it left product matrix} of the Riordan array $\R(g(x),f(x))$.

\begin{prop}\label{lp}
Let $R^L=\L(g(x),f(x))$ be the left product matrix of the Riordan array $R=\R(g(x),f(x))$.
Then
$$R=R^L\cdot\r.$$
\end{prop}
\begin{proof}
Let $R_k$ denote the $k$th column of $R$ and let $R_k(x)$ be the generating function of $R_k$.
Then $R_0(x)=g(x)$ and $R_{k+1}(x)=f(x)R_k(x)$.
Let $T=[f_{i-j}]$ be the Toeplitz matrix of the sequence $(f_n)_{n\ge 0}$.
Note that $R_{k+1}(x)=f(x)R_k(x)$ is equivalent to $R_{k+1}=T\cdot R_k$.
Hence
$$(R_1,R_2,R_3,\ldots)=(TR_0,TR_1,TR_2,\ldots)=T(R_0,R_1,R_2,\ldots)=TR.$$
Denote $\overline{R}=(R_1,R_2,R_3,\ldots)$. Then
$$R=(R_0,R_1,R_2,\ldots)=(R_0,\overline{R})=(R_0,TR)=
(R_0,T)\cdot\left(\begin{array}{cc}
1 & 0 \\
0 & R \\
\end{array}\right)
=R^L\cdot\r.$$
The proof is therefore complete.
\end{proof}

\begin{exm}\label{Cc-doc}
The coordination matrix $C=\R\left(\frac{1+x}{1-x},\frac{1+x}{1-x}\right)$
and the coordination triangle $\c=\R\left(\frac{1+x}{1-x},\frac{x(1+x)}{1-x}\right)$
have the decomposition:
\begin{equation*}\label{C-doc}
C=
\left(
  \begin{array}{ccccc}
    1 & 1 & 1 & 1 & \cdots \\
    2 & 4 & 6 & 8 &  \\
    2 & 8 & 18 & 32 &  \\
    2 & 12 & 38 & 88 &  \\
    \vdots &  &  &  & \ddots \\
  \end{array}
\right)
=
\left(
  \begin{array}{cccccc}
    1 & 1 &  &  & & \\
    2 & 2 & 1 &  & & \\
    2 & 2 & 2 & 1 & & \\
    2 & 2 & 2 & 2 & 1 &\\
    \vdots & \vdots & & &  & \ddots \\
  \end{array}
\right)\cdot
\left(
  \begin{array}{ccccc}
    1 & 0 & 0 & 0 &\cdots \\
    0 & 1 & 1 & 1 & \cdots \\
    0 & 2 & 4 & 6 &  \\
    0 & 2 & 8 & 18 &  \\
    \vdots & \vdots &  &  & \ddots \\
  \end{array}
\right)
\end{equation*}
and
\begin{equation*}\label{c-doc}
\c=\left(
  \begin{array}{ccccc}
    1 &  &  &  & \\
    2 & 1 &  &  &  \\
    2 & 4 & 1 &  &  \\
    2 & 8 & 6 & 1 &  \\
    \vdots &  &  &  & \ddots \\
  \end{array}
\right)
=
\left(
  \begin{array}{ccccc}
    1 &  &  &  & \\
    2 & 1 &  &  &  \\
    2 & 2 & 1 &  &  \\
    2 & 2 & 2 & 1 &  \\
    \vdots &  &  &  & \ddots \\
  \end{array}
\right)\cdot
\left(
  \begin{array}{ccccc}
    1 & 0 & 0 & 0 &\cdots \\
    0 & 1 & 0 & 0 & \cdots \\
    0 & 2 & 1 & 0 &  \\
    0 & 2 & 4 & 1 &  \\
    \vdots & \vdots &  &  & \ddots \\
  \end{array}
\right).
\end{equation*}
\end{exm}

We next consider the lower-diagonal-upper (LDU)
decomposition of coordination matrices $S$ and $C$.
Recall that the Delannoy numbers $D(n,k)=\sum_{i}\binom{n}{i}\binom{k}{i}2^i$.
Hence the Delannoy matrix $D=[D(n,k)]$ has the LDU decomposition:
$$D=\p\2p\p^t,$$
where $\2p$ is the diagonal matrix $\mathrm{diag}(1,2,2^2,2^3,\ldots)$,
$\p$ is the Pascal triangle and $\p^t$ is its transpose.
On the other hand,
$$S(n,k)=D(n,k)-D(n-1,k)=\sum_{i}\binom{n-1}{i-1}\binom{k}{i}2^i$$
and
$$C(n,k)=D(n,k)+D(n-1,k)=\sum_{i}\left[\binom{n}{i}+\binom{n-1}{i}\right]\binom{k}{i}2^i.$$

\begin{prop}\label{SC-doc}
We have the LDU decompositions:
$S=L_S\2p\p^t$ and $C=L_C\2p\p^t$, where
$$
L_S=\left[\binom{n-1}{k-1}\right]_{n\ge k\ge 0}
=\left(
  \begin{array}{ccccc}
    1 &  &  &  & \\
    0 & 1 &  &  & \\
    0 & 1 & 1 &  &  \\
    0 & 1 & 2 & 1 &  \\
    \vdots & \vdots &  &  & \ddots \\
  \end{array}
\right)
=\R\left(1,\frac{x}{1-x}\right)
$$
and
$$
L_C=\left[\binom{n}{k}+\binom{n-1}{k}\right]_{n\ge k\ge 0}
=
\left(
  \begin{array}{ccccc}
    1 &  &  &  &  \\
    2 & 1 &  &  &  \\
    2 & 3 & 1 &  &  \\
    2 & 5 & 4 & 1 &  \\
    \vdots &  &  &  & \ddots \\
  \end{array}
\right)
=\R\left(\frac{1+x}{1-x},\frac{x}{1-x}\right).
$$
\end{prop}

\begin{coro}
$\det [S(i,j)]_{0\le i,j\le n}
=\det [C(i,j)]_{0\le i,j\le n}
=2^{n(n+1)/2}.$
\end{coro}

\subsection{Coordination polynomials}

Define $\lm{m}(n,k)=L^{(m)}(n-k,k)$
and $\lmn{m}{n}(x)=\sum_{k=0}^n\lm{m}(n,k)x^k$.
Then by \eqref{Lnk-rr},
$$\lm{m}(n,k)=\lm{m}(n-1,k)+\lm{m}(n-1,k-1)+\lm{m}(n-2,k-1).$$
It follows that
\begin{equation}\label{lnx-rr}
\lmn{m}{n}(x)=(x+1)\lmn{m}{n-1}(x)+x\lmn{m}{n-2}(x)
\end{equation}
with $\lmn{m}{0}(x)=1$ and $\lmn{m}{1}(x)=x+m$.

Solve the recurrence relation \eqref{lnx-rr} by means of standard combinatorial techniques
to obtain the Binet form
\begin{equation}\label{lnx-b}
  \lmn{m}{n}(x)=\frac{\la_1+m-1}{\la_1-\la_2}\la_1^{n}-\frac{\la_2+m-1}{\la_1-\la_2}\la_2^{n},
\end{equation}
where
\begin{equation*}\label{la12}
\la_{1,2}=\frac{1+x\pm\sqrt{x^2+6x+1}}{2}
\end{equation*}
are the roots of the characteristic equation $\la^2-(x+1)\la-x=0$.

Clearly,
the Delannoy polynomials $d_n(x)=\lmn{1}{n}(x)$,
the coordination polynomials $c_n(x)=\lmn{2}{n}(x)$ and $s_n(x)=\lmn{0}{n}(x)$.
Hence
\begin{equation*}\label{dnx-b}
  d_n(x)=\frac{\la_1^{n+1}-\la_2^{n+1}}{\la_1-\la_2}
\end{equation*}
with $d_0(x)=1$ and $d_1(x)=x+1$.
Thus by \eqref{lnx-b}, we have
\begin{equation}\label{lnx-dnx}
\lmn{m}{n}(x)=d_n(x)+(m-1)d_{n-1}(x),\quad n=1,2,\ldots.
\end{equation}
In particular,
\begin{equation}\label{cnx-dnx}
c_n(x)=d_n(x)+d_{n-1}(x)
\end{equation}
and
\begin{equation}\label{snx-dnx}
s_n(x)=d_n(x)-d_{n-1}(x)
\end{equation}
for $n\ge 1$.
On the other hand,
$s_n(x)=xc_{n-1}(x)$ for $n\ge 1$.
It follows from \eqref{cnx-dnx} and \eqref{snx-dnx} that
\begin{equation}\label{cd-rr}
\left\{
  \begin{array}{ll}
    c_n(x)=xc_{n-1}(x)+2d_{n-1}(x), & \hbox{$c_0(x)=1$;} \\
    d_n(x)=xc_{n-1}(x)+d_{n-1}(x), & \hbox{$d_0(x)=1$.}
  \end{array}
\right.
\end{equation}

\begin{prop}\label{cnx}
The coordination polynomials $c_n(x)$ satisfy the recurrence relation
\begin{equation*}\label{cnx-rr-0}
c_n(x)=(x+1)c_{n-1}(x)+xc_{n-2}(x)
\end{equation*}
with $c_0(x)=1$ and $c_1(x)=x+2$,
and have the Binet form
\begin{equation*}\label{cnx-bf}
c_n(x)=\frac{\la_1+1}{\la_1-\la_2}\la_1^{n}-\frac{\la_2+1}{\la_1-\la_2}\la_2^{n}.
\end{equation*}
\end{prop}

\begin{rem}
  Both row sums $c_n(1)$ and $d_n(1)$ of $\c$ and $\d$ satisfy the recurrence relation
  $u_n=2u_{n-1}+u_{n-2}$.
  It is not difficult to obtain $c_n(1)$ and $d_n(1)$ are precisely the $n$th partial numerator and
  denominator of the continued fraction
  $$\sqrt{2}=1+\frac{1}{2+}\;\frac{1}{2+}\;\frac{1}{2+}\;\frac{1}{2+}\;\cdots,$$
  i.e., $\lim_{n\rightarrow +\infty}c_n(1)/d_n(1)=\sqrt{2}$.
\end{rem}

\section{Analytic properties of coordination numbers}

In this section we investigate analytic properties of coordination numbers,
including the total positivity of coordination matrices,
the distribution of zeros of the coordination polynomials,
the asymptotic normality of the coefficients of the coordination polynomials,
the log-concavity and the log-convexity of the coordination numbers.

\subsection{Total positivity of coordination matrices}

Let $A$ be a (finite or infinite) matrix of real numbers.
Following Pinkus \cite{Pin10},
we say that the matrix $A$ is {\it totally positive}
(TP for short)
if all its minors are nonnegative,
and {\it strictly totally positive}
(STP for short)
if all its minors are positive.
We use
$$A\left(
\begin{array}{cc}
i_0,\ldots,i_k \\
j_0,\ldots,j_k
\end{array}\right)$$
to denote the minor of the matrix $A$
determined by the rows indexed $i_0<i_1<\cdots<i_k$
and columns indexed $j_0<j_1<\cdots<j_k$ respectively.
For a (finite or infinite) lower triangular matrix $A$,
we say that it is {\it lower strictly totally positive} (LSTP for short)
if $A$ is totally positive and
$A\left(
\begin{array}{cc}
i_0,\ldots,i_k \\
j_0,\ldots,j_k
\end{array}\right)>0$
whenever $i_{\ell}\ge j_{\ell}$ for each $\ell$.

There have been a lot of recent interests in the total positivity of combinatorial matrices
(see \cite{CLW15,CLW15B,CW19,GP20,MMW22,Mon12a,Mon12e,WY18,Zhu21} for instance).
In this subsection we show that coordination matrices $S,\s,C,\c$ are all TP,
and furthermore, $C$ is STP and $\c$ is LSTP.

Let $(a_k)_{k\ge 0}$ be a (finite or infinite) sequence of nonnegative numbers
(we identify a finite sequence $a_0,a_1,\ldots,a_n$ with the infinite sequence $a_0,a_1,\ldots,a_n,0,0,\ldots$).
We say that the sequence is a {\it P\'olya frequency} sequence (PF for short)
if the corresponding Toeplitz matrix $[a_{i-j}]_{i,j\ge 0}$ is TP.
The following is a fundamental characterization for PF sequences (see~\cite[p. 412]{Kar68} for instance).

\begin{SE}
A sequence $(a_k)_{k\ge 0}$ of nonnegative numbers is PF if and only if its generating function has the form
\begin{equation*}\label{SE-fps}
  \sum_{k\ge 0}a_kx^k=ax^me^{\gamma x}\frac{\prod_{j\ge0}(1+\alpha_j x)}{\prod_{j\ge0} (1-\beta_j x)},
\end{equation*}
where $a>0, m\in\mathbb{N}, \alpha_j, \beta_j, \gamma\ge 0$
and $\sum_{j\ge0} (\alpha_j+\beta_j)<+\infty$.
\end{SE}

For example,
the sequence $\xi=(1,2,2,2,\ldots)$ is PF since it has generating function $\frac{1+x}{1-x}$,
and the corresponding Toplitz matrix
$$T(\xi)=
\left(\begin{array}{ccccc}
1 &  &  &  &\\
2 & 1 &  & & \\
2 & 2 & 1 & & \\
2 & 2 & 2 & 1 &\\
\vdots &  & & & \ddots \\
\end{array}\right)
$$
is therefore TP.

\begin{lem}[{\cite[Theorem 2.1]{CW19}}]\label{cw-c}
Let $f(x)=\sum_{n\ge 0}f_nx^n$ and $g(x)=\sum_{n\ge 0}g_nx^n$.
If both $(f_n)_{n\ge 0}$ and $(g_n)_{n\ge 0}$ are PF,
then the Riordan array $\R(g(x),f(x))$ is totally positive.
\end{lem}

An immediate consequence is the matrix
$L^{(m)}=\R\left(\frac{1+(m-1)x}{1-x},\frac{1+x}{1-x}\right)$
and the corresponding triangle $\widehat{L^{(m)}}=\R\left(\frac{1+(m-1)x}{1-x},\frac{x(1+x)}{1-x}\right)$
are all totally positive for $m\in\mathbb{N}$.
In particular,
the Delannoy matrix $D$ and the Delannoy triangle $\d$ are all totally positive,
which have been obtained by Brenti \cite[Corollary 5.15]{Bre95}.

\begin{coro}
The coordination matrices $S,\s,C,\c$ are all totally positive.
\end{coro}

We further show that the coordination matrix $C$ is STP
and the coordination triangle $\c$ is LSTP.
The following criterion for STP can be found in Pinkus \cite[Theorem 2.3]{Pin10}.

\begin{lem}\label{stp-c}
Let $A=[a_{ij}]_{0\le i,j\le m}$.
Then $A$ is STP if
$$A\binom{i,i+1,\ldots,i+k}{0,1,\ldots,k}>0
\quad\text{and}\quad
A\binom{0,1,\ldots,k}{i,i+1,\ldots,i+k}>0$$
for $k=0,1,2,\ldots,m$ and $i=0,\ldots,m-k$.
\end{lem}

The following criterion for LSTP is a transpose version of Pinkus \cite[Proposition 2.9]{Pin10}.
	
\begin{lem}
\label{LSTP-c}
Let $A=[a_{ij}]_{0\le i,j\le m}$ be a totally positive lower triangular matrix.
Then $A$ is LSTP if
$$A\binom{m-k,m-k+1,\ldots,m}{0,1,\ldots,k}>0$$
for $k=0,1,\ldots,m$.
\end{lem}

We also need the following classical result. 

\begin{CBF}
Let $A,B,C$ be three matrices and $C=AB$. Then
$$C\binom{i_1,\ldots,i_k}{j_1,\ldots,j_k}=
\sum_{\ell_1<\cdots<\ell_k}A\binom{i_1,\ldots,i_k}{\ell_1,\ldots,\ell_k}\cdot
B\binom{\ell_1,\ldots,\ell_k}{j_1,\ldots,j_k}.$$
\end{CBF}

\begin{rem}
It follows immediately
that the product of TP matrices is still TP.
\end{rem}

\begin{thm}\label{ff-c}
If $(f_n)_{n\ge 0}$ is PF and all $f_n>0$,
then
the Riordan square array $R=\R(f(x),f(x))$ is STP
and the Riordan triangular array $\hr=\R(f(x),xf(x))$ is LSTP.
\end{thm}
\begin{proof}
(1)\quad We show that $R=\R(f(x),f(x))$ is STP by Lemma \ref{stp-c}.
We need to show that
\begin{equation}\label{k-col}
  R\binom{i,i+1,\ldots,i+k}{0,1,\ldots,k}>0
\end{equation}
and
\begin{equation}\label{k-row}
  R\binom{0,1,\ldots,k}{i,i+1,\ldots,i+k}>0
\end{equation}
for $k\ge 0$ and $i\ge 0$.

By Proposition \ref{lp}, we have the decomposition $R=L\cdot\r$, where
$$L=
\left(
  \begin{array}{ccccc}
    f_0 & f_0 & 0 & 0 & \cdots \\
    f_1 & f_1 & f_0 & 0 &  \\
    f_2 & f_2 & f_1 & f_0 &  \\
    f_3 & f_3 & f_2 & f_1 &  \\
    \vdots & \vdots &  &  & \ddots \\
  \end{array}
\right),\quad
\r=
\left(
\begin{array}{cc}
    1 & 0 \\
    0 & R \\
  \end{array}
\right).$$
By the assumption that the sequence $(f_n)_{n\ge 0}$ is PF,
it follows that the Toeplitz matrix $T=[f_{i-j}]$ and the Riordan array $R=\R(f,f)$ are TP,
and so are the matrices $L$ and $\r$.

We first prove \eqref {k-col}.
We do this by induction on $k$.
The case $k=0$ is trivial, and so let $k\ge 1$.
Applying the Cauchy-Binet formula to $R=L\cdot\r$
and noting that minors of both $L$ and $\r$ are all nonnegative,
we have
\begin{eqnarray*}
R\binom{i,i+1,\ldots,i+k}{0,1,\ldots,k}
&=&\sum_{0\le j_0<j_1<\cdots<j_k}
L\binom{i,i+1,\ldots,i+k}{j_0,j_1,\ldots,j_k}\cdot \r\binom{j_0,j_1,\ldots,j_k}{0,1,\ldots,k}\\
&\ge& L\binom{i,i+1,\ldots,i+k}{0,i+2,\ldots,i+k+1}\cdot \r\binom{0,i+2,\ldots,i+k+1}{0,1,\ldots,k}\\
&=& (f_if_0^k)\cdot R\binom{i+1,i+2,\ldots,i+k}{0,1,\ldots,k-1}.
\end{eqnarray*}
It follows that \eqref{k-col} holds by induction on $k$.

We then prove \eqref{k-row}.
We do this by induction on $i$.
The case $i=0$ holds by \eqref{k-col}.
For $i\ge 1$,
we have
\begin{eqnarray*}
R\binom{0,1,\ldots,k}{i,i+1,\ldots,i+k}
&=&\sum_{0\le j_0<j_1<\cdots<j_k}
L\binom{0,1,\ldots,k}{j_0,j_1,\ldots,j_k}\cdot \r\binom{j_0,j_1,\ldots,j_k}{i,i+1,\ldots,i+k}\\
&\ge &L\binom{0,1,\ldots,k}{1,2,\ldots,k+1}\cdot \r\binom{1,2,\ldots,k+1}{i,i+1,\ldots,i+k}\\
&=&f_0^{k+1}\cdot R\binom{0,1,\ldots,k}{i-1,i,\ldots,i+k-1}.
\end{eqnarray*}
It follows that \eqref{k-row} holds by induction on $i$.

The Riordan array $R=\R(f(x),f(x))$ is therefore STP by Lemma \ref{stp-c}.

(2)\quad
Let $\hr_m$ be the $m$th leading principal submatrix of $\hr$ for $m\ge 0$.
To show that $\hr$ is LSTP,
it suffices to show that each $\hr_m$ is LSTP.
By Lemma \ref{LSTP-c},
we need to show that
for $0\le k\le m$,
\begin{equation}\label{c-m>0}
  \hr_m\binom{m-k,m-k+1,\ldots,m}{0,1,\ldots,k}>0.
\end{equation}
We proceed by induction on $m$.
By Proposition \ref{lp}, we have the decomposition $\hr=T\cdot\hr^+$, where
$$
T=\left(
  \begin{array}{ccccc}
    f_0 & 0 & 0 & 0 & \cdots \\
    f_1 & f_0 & 0 & 0 &  \\
    f_2 & f_1 & f_0 & 0 &  \\
    f_3 & f_2 & f_1 & f_0 &  \\
    \vdots &  &  &  & \ddots \\
  \end{array}
\right),\quad
\hr^+=
\left(
\begin{array}{cc}
    1 & 0 \\
    0 & \hr \\
  \end{array}
\right).$$
By the assumption that the sequence $(f_n)_{n\ge 0}$ is PF,
the Toeplitz matrix $T$ and the Riordan array $\hr$ are TP,
and so is the matrix $\hr^+$.
Applying the Cauchy-Binet formula to $\hr=T\cdot\hr^+$,
we have
\begin{eqnarray}\label{m-k}
\hr\binom{m-k,m-k+1,\ldots,m}{0,1,\ldots,k}
&=&
\sum_{0\le j_0<j_1<\cdots<j_k}
T\binom{m-k,m-k+1,\ldots,m}{j_0,j_1,\ldots,j_k}\cdot
\hr^+\binom{j_0,j_1,\ldots,j_k}{0,1,\ldots,k}\nonumber\\
&\ge &
T\binom{m-k,m-k+1,\ldots,m}{0,m-k+1,\ldots,m}\cdot
\hr^+\binom{0,m-k+1,\ldots,m}{0,1,\ldots,k}\nonumber\\
&=& (f_{m-k}f_0^{k})\cdot\hr\binom{m-k,\ldots,m-1}{0,1,\ldots,k-1}.
\end{eqnarray}
Note that for arbitrary $n\ge k\ge 0$,
$$\hr_n\binom{n-k,n-k+1,\ldots,n}{0,1,\ldots,k}
  =\hr\binom{n-k,n-k+1,\ldots,n}{0,1,\ldots,k}.$$
Hence by \eqref{m-k},
$$\hr_m\binom{m-k,m-k+1,\ldots,m}{0,1,\ldots,k}
\ge (f_{m-k}f_0^{k})\cdot \hr_{m-1}\binom{m-k,\ldots,m-1}{0,1,\ldots,k-1}.$$
Thus \eqref{c-m>0} holds by induction on $m$,
and $\hr$ is therefore LSTP.
\end{proof}

It follows immediately from Theorem \ref{ff-c} that
the Pascal matrix $P=\R\left(\frac{1}{1-x},\frac{1}{1-x}\right)$ is STP
and the Pascal triangle $\p=\R\left(\frac{1}{1-x},\frac{x}{1-x}\right)$ is LSTP.
Similarly, we have the following.

\begin{prop}\label{sc-s}
The coordination matrix $C=\R\left(\frac{1+x}{1-x},\frac{1+x}{1-x}\right)$ is STP
and the coordination triangle $\c=\R\left(\frac{1+x}{1-x},\frac{x(1+x)}{1-x}\right)$ is LSTP.
\end{prop}

\begin{coro}
The Delannoy matrix $D=\R\left(\frac{1}{1-x},\frac{1+x}{1-x}\right)$ is STP
and the Delannoy triangle $\d=\R\left(\frac{1}{1-x},\frac{x(1+x)}{1-x}\right)$ is LSTP.
\end{coro}
\begin{proof}
We show first that $D$ is STP by Lemma \ref{stp-c}.
Since $D$ is a symmetric matrix,
it suffices to show that $D\binom{i,i+1,\ldots,i+k}{0,1,\ldots,k}>0$ for $i\ge 0$ and $k\ge 0$.
Let
$$J=
\left(\begin{array}{ccccc}
1 &  &  &  &\\
1 & 1 &  & & \\
1 & 1 & 1 & & \\
1 & 1 & 1 & 1 &\\
\vdots &  & & & \ddots \\
\end{array}\right)
=\R\left(\frac{1}{1-x},x\right).$$
Then $D=J\cdot S$ by the definition \eqref{D-S}.
Note that both $J$ and $S$ are TP.
Hence
\begin{eqnarray*}
  D\binom{i,i+1,\ldots,i+k}{0,1,\ldots,k}
  &=& \sum_{0\le j_0<j_1<\cdots<j_k}
  J\binom{i,i+1,\ldots,i+k}{j_0,j_1,\ldots,j_k}
  \cdot
  S\binom{j_0,j_1,\ldots,j_k}{0,1,\ldots,k} \\
  &\ge & J\binom{i,i+1,\ldots,i+k}{0,i+1,\ldots,i+k}
  \cdot
  S\binom{0,i+1,\ldots,i+k}{0,1,\ldots,k}.
\end{eqnarray*}
Clearly,
$J\binom{i,i+1,\ldots,i+k}{0,i+1,\ldots,i+k}=1$,
and
$S\binom{0,i+1,\ldots,i+k}{0,1,\ldots,k}=C\binom{i,i+1,\ldots,i+k-1}{0,1,\ldots,k-1}>0$
since $C$ is STP by Proposition~\ref{sc-s}.
Thus $D\binom{i,i+1,\ldots,i+k}{0,1,\ldots,k}>0$,
and $D$ is therefore STP.

We show then that $\d$ is LSTP by Lemma \ref{LSTP-c}.
Now $\d=J\cdot\s$.
Hence for $0\le k\le m$,
\begin{eqnarray*}
\d\binom{m-k,m-k+1,\ldots,m}{0,1,\ldots,k}
&=&
\sum_{0\le j_0<j_1<\cdots<j_k}
J\binom{m-k,m-k+1,\ldots,m}{j_0,j_1,\ldots,j_k}\cdot
\s\binom{j_0,j_1,\ldots,j_k}{0,1,\ldots,k}\\
&\ge &
J\binom{m-k,m-k+1,\ldots,m}{0,m-k+1,\ldots,m}\cdot
\s\binom{0,m-k+1,\ldots,m}{0,1,\ldots,k}\\
&=&\c\binom{m-k,\ldots,m-1}{0,1,\ldots,k-1}.
\end{eqnarray*}
Thus $\d\binom{m-k,m-k+1,\ldots,m}{0,1,\ldots,k}>0$ since $\c$ is LSTP,
and $\d$ is therefore LSTP.
\end{proof}

Let $R^L$ be the left product matrix of a Riordan array $R$.
Very recently,
Mao~{\it et al.}~\cite[Theorem 3]{MMW22} showed that if $R^L$ is TP,
then so is $R$.
The following result can be obtained by means of
the same idea used in the proof of Theorem \ref{ff-c}.
We omit the details for the sake of brevity.

\begin{thm}\label{stp-gr}
Let $g(x)=\sum_{n\ge 0}g_nx^n$ and $f(x)=\sum_{n\ge 0}f_nx^n$ with $f_n>0$ and $g_n>0$ for all $n\ge 0$.
If the matrix $\L(g(x),f(x))$ is TP,
then the Riordan square matrix $\R(g(x),f(x))$ is STP;
if the triangle $\L(g(x),xf(x))$ is TP,
then the Riordan triangular matrix $\R(g(x),xf(x))$ is LSTP.
\end{thm}

\begin{rem}
Using Theorem \ref{stp-gr}
we can show that the triangle $\widehat{L^{(m)}}$ is LSTP when $m\ge 1$,
but we do not know in which cases the matrix $L^{(m)}$ is STP.
\end{rem}

\subsection{Zeros of coordination polynomials}

Following \cite{LW-rz},
let $\rz$ denote the set of real polynomials with only real zeros.
For $f\in\rz$ and $\deg f=n$,
let $r_n(f)\le\cdots\le r_2(f)\le r_1(f)$ denote the zeros of $f$.
Let $f,g\in\rz$ and $\deg f=n$.
We say that $g$ {\it interlaces} $f$, denoted by $g\sint f$,
if $\deg g=n-1$ and
\begin{equation*}\label{int}
  r_{n}(f)<r_{n-1}(g)<r_{n-1}(f)<\cdots<r_2(f)<r_1(g)<r_1(f).
\end{equation*}
We say that $g$ {\it alternates left of} $f$, denoted by $g\salt f$,
if $\deg g=n$ and
\begin{equation*}\label{alt}
  r_{n}(g)<r_{n}(f)<r_{n-1}(g)<r_{n-1}(f)<\cdots<r_2(f)<r_1(g)<r_1(f).
\end{equation*}
By $g\prec f$ we denote ``either $g\sint f$ or $g\salt f$".
For notational convenience,
let $a\prec bx+c$ for $a,c>0$ and $b\ge 0$.

Let $\sgn (r)$ denote the sign of the real number $r$, i.e.,
$$\sgn(r)=
\left\{
  \begin{array}{rl}
    +1, & \hbox{for $r>0$;} \\
    0, & \hbox{for $r=0$;} \\
    -1, & \hbox{for $r<0$.}
  \end{array}
\right.$$
The following result is folklore and obvious.

\begin{lem}\label{scr}
Let $f,g\in\rz$ and $\deg f=\deg g$ or $\deg g+1$.
Suppose that $f$ and $g$ have the positive leading coefficients.
Then $g\prec f$ if and only if one of the following two conditions holds:
\begin{enumerate}[\rm (i)]
  \item $\sgn f(r_i(g))=(-1)^i$ for $1\le i\le \deg g$.
  \item $\sgn g(r_i(f))=(-1)^{i-1}$ for $1\le i\le \deg f$.
\end{enumerate}
\end{lem}

It is known \cite{WZC19} that the zeros of the Delannoy polynomials $d_n(x)$ can be given explicitly by
\begin{equation}\label{dnx-z}
  r_{n,k}=r_k(d_n(x))=-\left(\sqrt{1+\cos^2\frac{k\pi}{n+1}}-\cos\frac{k\pi}{n+1}\right)^2,
  \quad k=1,2,\ldots,n.
\end{equation}
As a result, $d_{n-1}(x)\sint d_n(x)$
and
$$-3-2\sqrt{2}=-(\sqrt{2}+1)^2<r_{n,n}<\cdots<r_{n,k}<\cdots<r_{n,1}<-(\sqrt{2}-1)^2=-3+2\sqrt{2}.$$

\begin{thm}\label{cnx-z}
\begin{enumerate}[\rm (i)]
  \item
  All $c_n(x)$ have only real, simple zeros, and $c_{n-1}(x)\sint c_n(x)$.
  \item
  The zeros of $c_n(x)$ are in the open interval $(-3-2\sqrt{2},-3+2\sqrt{2})$.
\end{enumerate}
\end{thm}
\begin{proof}
(i)\quad
We proceed by induction on $n$.
Assume that $c_k(x)\in\rz$ for $k\le n$ and $c_{n-1}(x)\sint c_{n}(x)$.
We need to show that $c_{n+1}(x)\in\rz$ and $c_n(x)\sint c_{n+1}(x)$.

Let $s_{n,n}<\cdots<s_{n,1}$ be zeros of $c_n(x)$.
Then by $c_{n-1}(x)\sint c_{n}(x)$ and by Lemma \ref{scr}, we have
$$\sgn c_{n-1}(s_{n,i})=(-1)^{i-1},\quad i=1,2,\ldots,n.$$
It follows from $c_{n+1}(x)=(x+1)c_n(x)+xc_{n-1}(x)$ that
$$\sgn c_{n+1}(s_{n,i})=-\sgn c_{n-1}(s_{n,i})=(-1)^{i},\quad i=1,2,\ldots,n.$$
Thus $c_{n+1}(x)$ has $n+1$ real, simple zeros and $c_n(x)\sint c_{n+1}(x)$ again by Lemma \ref{scr}.

(ii)\quad
Let $r_{n+1,n+1}<\cdots<r_{n+1,1}$ be zeros of $d_{n+1}(x)$.
Then by $d_{n}(x)\sint d_{n+1}(x)$, we have
$$\sgn d_{n}(r_{n+1,i})=(-1)^{i-1},\quad i=1,2,\ldots,n+1.$$
Recall that $xc_n(x)=s_{n+1}(x)=d_{n+1}(x)-d_{n}(x)$.
We have
$$\sgn c_{n}(r_{n+1,i})=\sgn d_{n}(r_{n+1,i})=(-1)^{i-1},\quad i=1,2,\ldots,n+1.$$
Thus $c_{n}(x)\sint d_{n+1}(x)$ by Lemma \ref{scr}.
The zeros of the polynomial $d_{n+1}(x)$ are in the open interval $(-3-2\sqrt{2},-3+2\sqrt{2})$,
so are those of the polynomial $c_{n}(x)$.
This completes the proof.
\end{proof}

\begin{rem}
Using the same method,
it follows from $c_n(x)=d_n(x)+d_{n-1}(x)$ that
$$d_{n-1}(x)\sint c_n(x)\salt d_n(x).$$
\end{rem}

\begin{rem}
Note that $s_n(x)=xc_{n-1}(x)$.
Hence the zeros of the polynomial $s_n(x)$ are real, simple
and in the set $(-3-2\sqrt{2},-3+2\sqrt{2})\cup \{0\}$.
\end{rem}

\begin{rem}
For the row generating functions $\lmn{m}{n}(x)=\sum_{k=0}^n\lm{m}(n,k)x^k$
of the triangle $\widehat{L^{(m)}}$,
it follows from \eqref{lnx-rr} and \eqref{lnx-dnx} that
for $m\ge 2$, the polynomials
$\ell^{(m)}_n(x)$ has real, simple zeros and
$$x+m=\ell^{(m)}_{1}(x)\sint \ell^{(m)}_2(x)\sint\cdots\sint\ell^{(m)}_{n-1}(x)\sint\ell^{(m)}_n(x)\salt d_n(x).$$
The zeros of $\ell^{(m)}_n(x)$ are therefore
less than $-3+2\sqrt{2}$, and even less than $-3-2\sqrt{2}$ for large $m$.
\end{rem}

Let $(f_n(x))_{n\ge 0}$ be a sequence of complex polynomials.
We say that the complex number $x$ to be a {\it limit of zeros} of the sequence $(f_n(x))_{n\ge 0}$
if there is a sequence $(z_n)_{n\ge 0}$ such that $f_n(z_n)=0$ and $z\rightarrow x$ as $n\rightarrow +\infty$.
Suppose now that $(f_n(x))_{n\ge 0}$ is a sequence of polynomials satisfying the recursion 
\begin{equation*}
f_{n+k}(x)=-\sum_{j=1}^kp_j(x)f_{n+k-j}(x),
\end{equation*}
where $p_j(x)$ are polynomials in $x$.
Let $\la_j(x)$ be all roots of the associated characteristic equation
$\la^k+\sum_{j=1}^{k-1}p_j(x)\la^{k-j}=0$.
It is well known that if $\la_j(x)$ are distinct, then
\begin{equation}\label{r}
  f_n(x)=\sum_{j=1}^k\alpha_j(x)\la_j^n(x),
\end{equation}
where $\alpha_j(x)$ are determined from the initial conditions.

\begin{BKW}[{\cite[Theorem]{BKW78}}]\label{bkw}
Under the non-degeneracy requirements that in \eqref{r}
no $\alpha_j(x)$ is identically zero
and that for no pair $i\neq j$ is $\la_i(x)\equiv\omega\la_j(x)$ for some $\omega\in\mathbb{C}$ of unit modulus,
then $x$ is a limit of zeros of $(f_n(x))_{n\ge 0}$ if and only if either
\begin{enumerate}[\rm (i)]
  \item two or more of the $\la_i(x)$ are of equal modulus, and strictly greater (in modulus) than the others; or
  \item for some $j$, $\la_j(x)$ has modulus strictly than all the other $\la_i(x)$ have, and $\alpha_j(x)=0$.
\end{enumerate}
\end{BKW}

\begin{thm}
The zeros of all coordination polynomials $c_n(x)$ are dense in the closed interval $[-3-2\sqrt{2},-3+2\sqrt{2}]$.
\end{thm}
\begin{proof}
We prove a stronger result:
each $x\in [-3-2\sqrt{2},-3+2\sqrt{2}]$ is a limit of zeros of the sequence $(c_n(x))_{n\ge 0}$ of coordination polynomials.

Recall that the Binet form of the coordination polynomials:
\begin{equation}\label{cnx-b-1}
  c_n(x)=\frac{\la_1+1}{\la_1-\la_2}\la_1^{n}-\frac{\la_2+1}{\la_1-\la_2}\la_2^{n},
\end{equation}
where
\begin{equation*}\label{la12-1}
\la_{1,2}=\frac{1+x\pm\sqrt{x^2+6x+1}}{2}.
\end{equation*}

The non-degeneracy conditions of Beraha-Kahane-Weiss Theorem is clearly satisfied from \eqref{cnx-b-1}.
So the limits of zeros of $(c_n(x))_{n\ge 0}$ are those real numbers $x$ for which $|\la_1|=|\la_2|$,
i.e.,
$$|\sqrt{x^2+6x+1}+(x+1)|=|\sqrt{x^2+6x+1}-(x+1)|.$$
Thus $x^2+6x+1\le 0$.
Solve this inequality to obtain $$-3-2\sqrt{2}\le x\le -3+2\sqrt{2},$$
which is what we wanted to show.
\end{proof}

Let $a(n,k)$ be a double-indexed sequence of nonnegative numbers and let
\begin{equation*}\label{pnk}
p(n,k)=\frac{a(n,k)}{\sum_{j=0}^na(n,j)}
\end{equation*}
denote the normalized probabilities.
Following Bender~\cite{Ben73},
we say that the sequence $a(n,k)$ is {\it asymptotically normal by a central limit theorem},
if
\begin{equation}\label{clt}
\lim_{n\rightarrow\infty}\sup_{x\in\mathbb{R}}\left|\sum_{k\le\mu_n+x\sigma_n}p(n,k)
-\frac{1}{\sqrt{2\pi}}\int_{-\infty}^xe^{-t^2/2}dt\right|=0,
\end{equation}
where $\mu_n$ and $\sigma^2_n$ are the mean and variance of $a(n,k)$, respectively.
We say that $a(n,k)$ is {\it asymptotically normal by a local limit theorem} on $\mathbb{R}$ if
\begin{equation}\label{llt}
\lim_{n\rightarrow\infty}\sup_{x\in\mathbb{R}}\left|\sigma_np(n,\lfloor\mu_n+x\sigma_n\rfloor)-\frac{1}{\sqrt{2\pi}}e^{-x^2/2}\right|=0.
\end{equation}
In this case,
\begin{equation*}\label{asy}
a(n,k)\sim \frac{e^{-x^2/2}\sum_{j=0}^na(n,j)}{\sigma_n\sqrt{2\pi}} \textrm{ as } n\rightarrow \infty,
\end{equation*}
where $k=\mu_n+x\sigma_n$ and $x=O(1)$.
Clearly, the validity of \eqref{llt} implies that of \eqref{clt}.

Many well-known combinatorial sequences enjoy central and local limit theorems.
For example, the famous de Movior-Laplace theorem states that
the binomial coefficients $\binom{n}{k}$
are asymptotically normal (by central and local limit theorems).
Other examples include
the signless Stirling numbers $c(n,k)$ of the first kind,
the Stirling numbers $S(n,k)$ of the second kind,
and the Eulerian numbers $A(n,k)$ (see \cite{Can15} for instance).
A standard approach to demonstrating asymptotic normality is the following criterion
(see \cite[Theorem 2]{Ben73} for instance).

\begin{lem}\label{lem-rzv}
Suppose that $A_n(x)=\sum_{k=0}^na(n,k)x^k$ have only real zeros
and $A_n(x)=\prod_{i=1}^n(x+t_{n,i})$,
where all $t_{n,i}$ are nonnegative.
Let $$\mu_n=\sum_{i=1}^n\frac{1}{1+t_{n,i}}$$
and
\begin{equation}\label{s-2}
\sigma^2_n=\sum_{i=1}^n\frac{t_{n,i}}{(1+t_{n,i})^2}.
\end{equation}
Then if $\sigma_n^2\rightarrow+\infty$ as $n\rightarrow+\infty$,
the numbers $a(n,k)$ are asymptotically normal (by central and local limit theorems)
with the mean $\mu_n$ and variance $\sigma_n^2$.
\end{lem}

\begin{prop}\label{an-tT}
Suppose that $a(n,k)$ are nonnegative
and $A_n(x)=\sum_{k=0}^na(n,k)x^k$ have only real zeros.
If there are two positive numbers $t_0$ and $t_1$ such that
zeros of all $A_n(x)$ are in the interval $(-t_1,-t_0)$.
then the numbers $a(n,k)$ are asymptotically normal (by central and local limit theorems).
\end{prop}
\begin{proof}
Since $t_{n,i}\in (t_0,t_1)$, we have by \eqref{s-2}
$$
\sigma^2_n=\sum_{i=1}^n\frac{t_{n,i}}{(1+t_{n,i})^2}
>\sum_{i=1}^n\frac{t_0}{(1+t_1)^2}=\frac{t_0}{(1+t_1)^2}n.
$$
Thus $\sigma_n^2\rightarrow+\infty$ as $n\rightarrow+\infty$,
and the numbers $a(n,k)$ are therefore asymptotically normal (by central and local limit theorems)
by Lemma \ref{lem-rzv}.
\end{proof}

Recall that all zeros of $c_n(x)$
are real and in the interval $(-3-2\sqrt{2},-3+2\sqrt{2})$.
So the following result is immediate.

\begin{coro}\label{cnk-an}
The coordination numbers $c(n,k)$
are asymptotically normal (by central and local limit theorems).
\end{coro}

\begin{rem}
Similarly,
the Delannoy numbers $d(n,k)$ are asymptotically normal,
which has been obtained in \cite[Theorem 3.2]{WZC19}
by means of the explicit expression \eqref{dnx-z} of the zeros of $d_n(x)$.

On the other hand,
$s_n(x)=xc_{n-1}(x)$,
the coordination numbers $s(n,k)$ are also asymptotically normal by Lemma \ref{lem-rzv}.
\end{rem}

\begin{rem}\label{c-mean}
Note that in Lemma \ref{lem-rzv},
the mean $\mu_n$ has alternative expression $A'_n(1)/A_n(1)$.
Now the associated means of the numbers $d(n,k)$ is $n/2$
since $d(n,k)=d(n,n-k)$ for $0\le k\le n$.
Hence $d'_n(1)/d_n(1)=n/2$.
Recall that
$c_n(x)=d_n(x)+d_{n-1}(x)$
and $c_n(1)/d_n(1)\rightarrow\sqrt{2}$.
The associated means of the numbers $c(n,k)$
$$\mu_n=\frac{c'_n(1)}{c_n(1)}
=\frac{d'_n(1)+d'_{n-1}(1)}{d_n(1)+d_{n-1}(1)}
=\frac{\frac{n}{2}d_n(1)+\frac{n-1}{2}d_{n-1}(1)}{d_n(1)+d_{n-1}(1)}
=\frac{n-1}{2}+\frac{1}{2}\frac{d_n(1)}{c_n(1)}
=\frac{n-1}{2}+\frac{1}{2\sqrt{2}}+\mathrm{o}(1).$$
Similarly, the associated means of the numbers $s(n,k)$
$$\mu_n=\frac{s'_n(1)}{s_n(1)}=\frac{c'_{n-1}(1)+c_{n-1}(1)}{c_{n-1}(1)}
=\frac{c'_{n-1}(1)}{c_{n-1}(1)}+1=\frac{n}{2}+\frac{1}{2\sqrt{2}}+\mathrm{o}(1).$$
\end{rem}

\subsection{Log-concavity and log-convexity of coordination sequences}

Let $(a_k)_{k\ge 0}$ be a sequence of nonnegative numbers.
We say that the sequence is {\it log-concave} if $a_{k-1}a_{k+1}\le a_k^2$ for $k\ge 1$,
and {\it unimodal} if
$$a_0\le a_1\le\cdots\le a_m\ge a_{m+1}\ge\cdots$$
for some $m$ ($m$ is called a {\it mode} of the sequence).
It is well known \cite{Bre89} that
a P\'olya frequency sequence is log-concave, and that
a log-concave sequence with no internal zeros is unimodal.

\begin{prop}
For each $k\ge 0$,
the sequence
$C(0,k),C(1,k),C(2,k),\ldots$
is 
log-concave.
\end{prop}
\begin{proof}
The sequence is the $k$th column of the coordination matrix $C$ and
has the generating function
$$\sum_{n\ge 0}C(n,k)x^k=\left(\frac{1+x}{1-x}\right)^k.$$
Thus the sequence is PF and therefore log-concave.
\end{proof}

\begin{prop}
For each $n\ge 1$,
the sequence $C(n,0),C(n,1),C(n,2),\ldots$ is log-concave.
\end{prop}
\begin{proof}
The sequence is the $n$th row of the coordination matrix $C$ and
has the generating function
$$\sum_{k\ge 0}C(n,k)y^k=[x^n]C(x,y)=[x^n]\left(\frac{1+x}{1-x-y-xy}\right)=\frac{2(1+y)^{n-1}}{(1-y)^{n+1}}.$$
Thus the sequence is PF and therefore log-concave.
\end{proof}

The following is a classical approach for attacking the unimodality and log-concavity problem of a finite sequence
(see \cite{Bre89} for instance).

\begin{NI}
Suppose that the polynomial $f(x)=\sum_{k=0}^na_kx^k$ has only real zeros. Then
$$a_k^2\ge a_{k-1}a_{k+1}\frac{(k+1)(n-k+1)}{k(n-k)},\qquad k=1,2,\ldots,n-1.$$
\end{NI}

\begin{rem}
If all coefficients $a_k$ are nonnegative,
then the sequence $a_0,a_1,\ldots,a_n$ 
is log-concave and unimodal with at most two modes.
Darroch \cite{Dar64} showed that each mode $m$ of such a sequence satisfies
$\left|m-f'(1)/f(1)\right|<1$.
\end{rem}

We have proved that the coordination polynomials $c_n(x)=\sum_{k=0}^nc(n,k)x^k$ have only real zeros.
The sequence $c(n,0),c(n,1),\ldots,c(n,n)$ is therefore log-concave and unimodal with at most two modes.
Furthermore, each mode $m$ satisfies $\left|m-c'_n(1)/c_n(1)\right|<1$.
It follows from Remark \ref{c-mean} that $\lrf{(n-1)/2}\le m\le\lrc{(n+1)/2}$.
We have the following precise result.

\begin{prop}
For each $n\ge 1$, the sequence $C(n,0),C(n-1,1),C(n-2,2),\ldots,C(0,n)$
is log-concave and unimodal with the unique mode $\lrf{n/2}$.
\end{prop}
\begin{proof}
Note that $C(n-k,k)=c(n,k)$ for $k=0,1,\ldots,n$.
Hence the sequence
is log-concave and unimodal.
We next show that the sequence 
has the unique mode $\lrf{n/2}$.
We prove this only for the case $n$ odd, since the proof is similar for the case $n$ even.
Let $n=2m+1$.
Note that
$$d(n,0)<\cdots<d(n,m-1)<d(n,m)=d(n,m+1)>\cdots>d(n,n-1)>d(n,n)$$
and
$$d(n-1,0)<\cdots<d(n-1,m-1)=d(n-1,m)>d(n-1,m+1)>\cdots>d(n-1,n-1).$$
Also, $c(n,k)=d(n,k)+d(n-1,k)$ for $0\le k\le n-1$ and $c(n,n)=d(n,n)$.
Hence
$$c(n,0)<\cdots<c(n,m-1)<c(n,m)>c(n,m+1)>\cdots>c(n.n-1)>c(n,n).$$
Thus the sequence $c(n,0),c(n,1),\ldots,c(n,n)$ has the unique mode $m$, as required.
\end{proof}

We say that a sequence $(a_n)_{n\ge 0}$ of positive numbers is
{\it log-convex} if $a_{n-1}a_{n+1}\ge a_n^2$ for $n\ge 1$,
and {\it strictly log-convex} if $a_{n-1}a_{n+1}>a_n^2$ for $n\ge 1$.
It is well known that the sequence $(D_n)_{n\ge 0}$ of the central Delannoy numbers is log-convex
(see \cite{LW-lcx} for instance).
We next consider the log-convexity of the diagonal sequences
$(S_n)_{n\ge 0}$ and $(C_n)_{n\ge 0}$
in coordination matrices $S$ and $C$.
Recall that the sequences $(D_n)_{n\ge 0}$, $(S_n)_{n\ge 0}$ and $(C_n)_{n\ge 0}$
satisfy a three-term recurrence relation respectively.
We establish a criterion for the log-convexity of such sequences.

\begin{lem}\label{3rr-lcx}
Let $(z_n)_{n\ge 0}$ be a sequence of positive numbers and satisfy the recurrence relation
$$a_nz_n=b_nz_{n-1}-c_nz_{n-2},$$
where $a_n>0$ and $b_n,c_n\ge0$ for $n\ge 2$.
Assume that $z_1>z_0$ and $z_0z_2>z_1^2$.
If
$$\left|
\begin{array}{cc}
a_n & a_{n+1}\\
b_n & b_{n+1}\\
\end{array}
\right|\ge
\max\left\{0,
\left|
\begin{array}{cc}
a_n & a_{n+1}\\
c_n & c_{n+1}\\
\end{array}
\right|\right\}$$
for $n\ge 2$,
then the sequence $(z_n)_{n\ge 0}$ is strictly log-convex.
\end{lem}
\begin{proof}
Let $x_n={z_n}/{z_{n-1}}$ for $n\ge 1$.
Then $x_n=\frac{b_n}{a_n}-\frac{c_n}{a_n}\cdot\frac{1}{x_{n-1}}$ for $n\ge 2$.
To prove the strict log-convexity of the sequence $(z_n)_{n\ge 0}$,
it suffices to prove that the sequence $(x_n)_{n\ge 1}$ is increasing.
We do this by showing that $x_{n+1}>x_n>1$ by induction on $n\ge 1$.

Clearly, $x_2>x_1>1$ by the assumption $z_1>z_0$ and $z_0z_2>z_1^2$.
Assume now that $x_m>x_{m-1}>1$ for some $m\ge 2$.
Then
\begin{eqnarray*}
x_{m+1}-x_m
&=& \left[\frac{b_{m+1}}{a_{m+1}}-\frac{c_{m+1}}{a_{m+1}}\cdot\frac{1}{x_{m}}\right]
-\left[\frac{b_m}{a_m}-\frac{c_m}{a_m}\cdot\frac{1}{x_{m-1}}\right] \\
&=& \frac{\left[(a_mb_{m+1}-a_{m+1}b_m)x_m-(a_mc_{m+1}-a_{m+1}c_m)\right]}{a_ma_{m+1}x_m}
+\frac{c_m}{a_m}\left(\frac{1}{x_{m-1}}-\frac{1}{x_m}\right)\\
&\ge & \frac{(a_mb_{m+1}-a_{m+1}b_m)-(a_mc_{m+1}-a_{m+1}c_m)}{a_ma_{m+1}x_m}
+\frac{c_m}{a_m}\cdot\frac{x_m-x_{m-1}}{x_{m-1}x_m}.
\end{eqnarray*}
Thus $x_{m+1}-x_m>0$,
and so $x_{m+1}>x_m>1$, as required.
\end{proof}

\begin{prop}
The sequences $(D_n)_{n\ge 0}, (S_n)_{n\ge 0}$ and $(C_n)_{n\ge 0}$ are all strictly log-convex.
\end{prop}
\begin{proof}
Recall that
\begin{enumerate}[\rm (i)]
  \item $nD_n=3(2n-1)D_{n-1}-(n-1)D_{n-2}$ with $D_0=1$ and $D_1=3$.
  \item $n(2n-3)S_n=4(3n^2-6n+2)S_{n-1}-(n-2)(2n-1)S_{n-2}$ with $S_0=1$ and $S_1=2$.
  \item $n(n-1)C_n=3(2n-1)(n-1)C_{n-1}-n(n-2)C_{n-2}$ with $C_0=1$ and $C_1=4$.
\end{enumerate}
The statement follows from Lemma \ref{3rr-lcx}.
\end{proof}

A concept closely related to log-convex sequences is Stieltjes moment sequences.
Let $\alpha=(a_n)_{n\ge 0}$ be an infinite sequence of real numbers.
If there exist a nonnegative Borel measure $\mu$ on $[0,+\infty)$ such that
\begin{equation*}\label{h-i-e}
a_n=\int_{0}^{+\infty}x^nd\mu(x),
\end{equation*}
then we say that $\alpha=(a_n)_{n\ge 0}$ is a {\it Stieltjes moment} sequence (SM for short).
It is well known that a Stieltjes moment sequence is log-convex (see \cite{WZ16} for instance).

Let
$$H(\alpha)=[a_{i+j}]_{i,j\ge 0}$$
be the {\it Hankel matrix} of the sequence $\alpha$
and let
$$h_n(\alpha)=\det [a_{i+j}]_{0\le i,j\le n}$$
denote the $n$th leading principal minor of $H(\alpha)$.
The following result is well known (see \cite{Ben11} for instance).

\begin{lem}\label{Ben-lem}
Let $\alpha=(a_n)_{n\ge 0}$ and $\overline{\alpha}=(a_{n+1})_{n\ge 0}$.
\begin{enumerate}[\rm (i)]
  \item
  The sequence $\alpha$ is SM if and only if
  its Hankel matrix $H(\alpha)$ is TP.
  \item
  The sequence $\alpha$ is SM if and only if
  both $h_n(\alpha)\ge 0$ and $h_n(\overline{\alpha})\ge 0$ for all $n\ge 0$.
  \item
  If $\alpha=(a_n)_{n\ge 0}$ is SM,
  then so is $\overline{\alpha}=(a_{n+1})_{n\ge 0}$.
  \item
  If both $(a_n)_{n\ge 0}$ and $(b_n)_{n\ge 0}$ are SM,
  then so is $(aa_n+bb_n)_{n\ge 0}$ for arbitrary $a,b\ge 0$.
\end{enumerate}
\end{lem}

The sequence $(D_n)_{n\ge 0}$ of the central Delannoy numbers is SM.
Actually, the following result is known (see \cite{MWY17} for instance).

\begin{lem}\label{Dn-s}
Let $\delta=(D_n)_{n\ge 0}$ 
and $\overline{\delta}=(D_{n+1})_{n\ge 0}$.
Then
\begin{equation}\label{Dn-h}
  h_n(\delta)=2^{n(n+3)/2},\qquad h_n(\overline{\delta})=2^{n(n+3)/2}(2^{n+1}+1).
\end{equation}
\end{lem}

The sequence
$(C_n)_{n\ge 0}$ is not SM
since the Hankel determinant $\det [C_{i+j}]_{0\le i,j\le 2}=-4<0$.
We next show that the sequence $(S_n)_{n\ge 0}$ is SM.
We need the following determinant evaluation rule (see Zeilberger \cite{Zil97} for a combinatorial proof).

\begin{DJI}\label{dji}
Let the matrix $M=[m_{ij}]_{0\le i,j\le k}$. 
Then $$\det M\cdot\det M^{0,k}_{0,k}=\det M_k^k\cdot\det M_0^0-\det M_0^k\cdot\det M_k^0,$$
where $M^I_J$ denote the submatrix obtained from $M$ by deleting those rows in $I$ and columns in $J$.
\end{DJI}

\begin{thm}
The coordination numbers $S_0,S_1,S_2,\ldots$ form a Stieltjes moment sequence.
\end{thm}
\begin{proof}
Let $E_0=1, E_n=D_{n-1}$ for $n\ge 1$ and denote $\varepsilon=(E_n)_{n\ge 0}$.
Then by \eqref{s-d-d}, we have
$$S_{n}=\frac{1}{2}(D_{n}+E_{n}),\qquad n=0,1,2,\ldots.$$
By Lemma \ref{Ben-lem} (iv) and Lemma \ref{Dn-s},
to show that $(S_n)_{n\ge 0}$ is SM,
it suffices to show that $(E_n)_{n\ge 0}$ is SM.
By Lemma~\ref{Ben-lem}~(ii),
we need to show that $h_n(\varepsilon)\ge 0$ and $h_n(\overline{\varepsilon})\ge 0$ for all $n\ge 0$.
Note that $\overline{\varepsilon}=\delta$ and $h_n(\delta)>0$ by \eqref{Dn-h}.
Hence it suffices to show that $h_n(\varepsilon)\ge 0$.
Actually, we can show that $h_n(\varepsilon)=2^{n(n+1)/2}$.

Applying Desnanot-Jacobi Determinant Identity to the Hankel matrix $[E_{i+j}]_{0\le i,j\le k}$, we obtain
$$h_k(\varepsilon)\cdot h_{k-2}(\overline{\delta})
=h_{k-1}(\varepsilon)\cdot h_{k-1}(\overline{\delta})-\left[h_{k-1}(\delta)\right]^2,$$
or equivalently,
$$\frac{h_k(\varepsilon)}{h_{k-1}(\overline{\delta})}=
\frac{h_{k-1}(\varepsilon)}{h_{k-2}(\overline{\delta})}
-\frac{\left[h_{k-1}(\delta)\right]^2}{h_{k-2}(\overline{\delta})h_{k-1}(\overline{\delta})}.$$
It follows that
\begin{equation*}\label{h-f}
  \frac{h_n(\varepsilon)}{h_{n-1}(\overline{\delta})}=
\frac{h_1(\varepsilon)}{h_{0}(\overline{\delta})}
-\sum_{k=2}^{n}\frac{\left[h_{k-1}(\delta)\right]^2}{h_{k-2}(\overline{\delta})h_{k-1}(\overline{\delta})}.
\end{equation*}
Note that $h_1(\varepsilon)=2, h_{0}(\overline{\delta})=3$ and
$$
\frac{\left[h_{k-1}(\delta)\right]^2}{h_{k-2}(\overline{\delta})h_{k-1}(\overline{\delta})}
=\frac{2^{k}}{(2^{k-1}+1)(2^k+1)}
=\frac{2}{2^{k-1}+1}-\frac{2}{2^{k}+1}
$$
by \eqref{Dn-h}.
Hence
$$\frac{h_n(\varepsilon)}{h_{n-1}(\overline{\delta})}
=\frac{2}{3}-\sum_{k=2}^{n}\left[\frac{2}{2^{k-1}+1}-\frac{2}{2^{k}+1}\right]
=\frac{2}{2^{n}+1},$$
and so
$$h_n(\varepsilon)=\frac{2}{2^{n}+1}h_{n-1}(\overline{\delta})
=2^{n(n+1)/2}$$
by \eqref{Dn-h} again,
as claimed.
The sequence $(S_n)_{n\ge 0}$ is therefore SM.
\end{proof}

\section{Concluding remarks and further work}

Conway and Sloane \cite{CS97} investigated the coordination sequences
of many classical root lattices besides the cubic lattices
and gave explicit formulae for these coordination sequences.
A natural problem is to explore analytic behaviors of
the corresponding coordination sequences.

The coordination sequences of root lattices are studied usually
by means of their coordinator triangles and coordinator polynomials.
There have been some results related to analytic properties of
coordinator triangles and coordinator polynomials
of root lattices.
For example,
Conway and Sloane \cite{CS97} showed that the coordinator triangle for the lattices of type A
is precisely the Narayana triangle
$\mathrm{N}_\mathrm{B}=\left[\binom{n}{k}^2\right]_{n\ge k\ge 0}$
of type B.
The Narayana triangle $\mathrm{N}_\mathrm{B}$ is totally positive~\cite{WY18}
and its row generating functions have only real zeros~\cite{WZ13}. 
It is not difficult to show that the numbers $\binom{n}{k}^2$
are asymptotically normal 
by Stirling approximation.
Conway and Sloane \cite{CS97} showed also that the coordinator polynomials of the dual lattices D$^*$ of type D
have the form of $(x+1)A_{n-1}(x)$,
where $A_n(x)$ are the classical Eulerian polynomials.
Thus
the coordinator triangle $\mathrm{CT(D}^*)$ of D$^*$ 
has the matrix decomposition $\mathrm{CT(D}^*)=A^+\cdot R_0$,
where $A$ is the Eulerian triangle,
$A^+=
\left(
  \begin{array}{cc}
    1 & 0 \\
    0 & A \\
  \end{array}
\right)$ and $R_0=\R(1+x,x)$, i.e.,
$$
\left(
  \begin{array}{cccccc}
    1 &  &  &  & &\\
    1 & 1 &  &  &  &\\
    1 & 2 & 1 &  &  &\\
    1 & 5 & 5 & 1 &  &\\
    1 & 12 & 22 & 12 & 1 &\\
    \vdots &  &  & & & \ddots \\
  \end{array}
\right)
=
\left(
  \begin{array}{cccccc}
    1 &  &  &  & &\\
    0 & 1 &  &  &  & \\
    0 & 1 & 1 &  & & \\
    0 & 1 & 4 & 1 & & \\
    0 &1 &11& 11 &1 &\\
    \vdots & \vdots &  & & & \ddots \\
  \end{array}
\right)\cdot
\left(
  \begin{array}{cccccc}
    1 &  &  &  & &\\
    1 & 1 &  &  &  &\\
     & 1 & 1 &  &  &\\
     &  & 1 & 1 &  &\\
     & & & 1 & 1 &  \\
    &  & & &\ddots & \ddots\\
  \end{array}
\right)
.$$
The Eulerian polynomials have only real zeros (see \cite{WY05} for instance),
so do the coordinator polynomials of D$^*$.
The Eulerian numbers are asymptotically normal (see \cite{CKSS72} for instance),
so are the coordinator numbers of D$^*$ by Lemma \ref{lem-rzv}
(see \cite[\S 4.5.5]{HCD20} also).
However, we do not know whether the coordinator triangle of D$^*$ is TP.
A closely related problem is a long-standing conjecture due to Brenti \cite[Conjecture 6.10]{Bre96},
which claims that
the Eulerian triangle $A$ is TP.
Clearly, $R_0$ is TP.
It is also clear that $A^+$ is TP if and only if $A$ is TP
and that the product of TP matrices is still TP.
So, if Brenti's conjecture is true,
then the coordinator triangle $\mathrm{CT(D}^*)$ is TP.
We refer the reader to \cite{ABHPS11,BG97,Mon15,PS07,WZ13,XZ14}
for some more related results.
An interesting topic is to systematically investigate
analytic properties of
coordinator triangles and coordinator polynomials
of root lattices.


\section*{Acknowledgements}

This work was supported in part by the
National Natural Science Foundation of China (Nos. 12001301,12171068).

\end{document}